\documentclass[12pt,a4paper]{article}
\usepackage[margin=1in]{geometry}
\usepackage{amsmath,amssymb,amsfonts,amsthm,mathtools}
\usepackage[numbers,sort&compress]{natbib}
\usepackage{xcolor}
\usepackage{hyperref}

\definecolor{mycyanblue}{HTML}{0080AC}
\hypersetup{colorlinks=true,allcolors=mycyanblue}

\theoremstyle{plain}
\newtheorem{theorem}{Theorem}[section]
\newtheorem{lemma}[theorem]{Lemma}
\newtheorem{proposition}[theorem]{Proposition}
\newtheorem{corollary}[theorem]{Corollary}
\theoremstyle{definition}
\newtheorem{definition}[theorem]{Definition}
\newtheorem{remark}[theorem]{Remark}

\usepackage{tikz}
\usetikzlibrary{positioning,arrows.meta,calc}

\title{A Characterization of Walk-Matrix Equivalence at Corank Two via Reciprocal WQH Switching}

\author{
 Chaochao Zhu\textsuperscript{a*} \quad
 Qin Yue\textsuperscript{a}\\
{\small
\textsuperscript{a}College of Finance and Mathematics, West Anhui University, Lu'an 237012, China}
}
\date{}

\begin{document}
\maketitle

\begingroup
\renewcommand{\thefootnote}{}
\footnotetext{
*\,Corresponding author. \\
\texttt{E-mail: zccjsbz@amss.ac.cn(C.Zhu), yqer@163.com(Q.Yue)} 
}
\endgroup

\begin{abstract}
Let $G$ be a graph of order $n$ with adjacency matrix $A_G$, let $\mathbf e$ denote the all-one vector, and let \(W_G=[\mathbf e,A_G\mathbf e,\ldots,A_G^{n-1}\mathbf e]\)
be its walk matrix. We consider the case \(\operatorname{rank}W_G=n-2,\) the first corank for which distinct graphs can have the same walk matrix. We give a complete structural description of such pairs. More precisely, if $G$ and $H$ are distinct graphs on the same labelled vertex set and \(\operatorname{rank}W_G=n-2\), then
\(W_G=W_H\) if and only if $H$ is obtained from $G$ by a reciprocal Wang--Qiu--Hu (WQH) switching. In this case, \(A_G-A_H=uv^T+vu^T,\) where \(u,v\in\{0,\pm1\}^n\) have disjoint supports and form a basis of \(\ker W_G^T\). We also determine the minimum order at which a non-isomorphic corank-two walk mate can occur: no such pair exists for 
\(n\le9\), so the previously known $10$-vertex example is sharp. Starting from a labelled realization of that pair, we use singleton union and join operations, together with the graph coronal, to construct connected non-isomorphic pairs with equal corank-two walk matrices for every \(n\ge10\).
\end{abstract}


\section{Introduction}\label{sec:introduction}

Let $G$ be a simple graph on $V(G)=\{1,\ldots,n\}$, with adjacency matrix $A_G$, and let \(\mathbf e\) denote the all-one vector. Its standard \emph{walk matrix} is
\[
W_G:=W(G)=
[\mathbf e,A_G\mathbf e,A_G^2\mathbf e,\ldots,
A_G^{n-1}\mathbf e].
\]
More generally, for an \(n\times n\) matrix $M$, we write
\[
W(M):=
[\mathbf e,M\mathbf e,M^2\mathbf e,\ldots,M^{n-1}\mathbf e].
\]
Thus \(W_G=W(A_G).\) Whenever two walk matrices are compared, the corresponding graphs are regarded as graphs on the same labelled vertex set. The $i$-th entry of \(A_G^k\mathbf e\) is the number of walks of length $k$ starting at vertex $i$. Thus $W_G$ records, vertex by vertex, the successive walk counts. Standard references for spectral and matrix methods in graph theory include~\cite{Bapat2014,BrouwerHaemers2012}.

The column space \(\mathcal K_G:=\operatorname{col}W_G\) is the Krylov space generated by \(\mathbf e\). Equivalently, if \(E_\lambda\) denotes the spectral idempotent associated with an eigenvalue \(\lambda\) of $A_G$, then
\[
\mathcal K_G
=\operatorname{span}
\{E_\lambda\mathbf e:E_\lambda\mathbf e\neq0\},
\]
where the corresponding \(\lambda\) are precisely the distinct main eigenvalues of \(G\); see~\cite{ScirihaCollins2023}. In particular, \(\operatorname{rank}W_G\) equals the number of distinct main eigenvalues; see, for example, \cite{Hagos2002,LiuSiemons2022}. A graph with nonsingular walk matrix is usually called \emph{controllable}~\cite{Godsil2012}. Walk matrices also play a central role in generalized spectral characterizations; see, for instance, \cite{Wang2013,Wang2017,WangXu2006}. For broader surveys on spectral determination, we refer to \cite{vanDamHaemers2003,vanDamHaemers2009}.

The inverse problem considered here is basic: how much of the graph is already encoded in $W_G$? Liu and Siemons~\cite{LiuSiemons2022} proved that
\(\operatorname{rank}W_G\geq n-1\) forces the adjacency matrix. Hence \(\operatorname{rank}W_G=n-2\) is the first place where uniqueness can fail. They showed that, in
this case, the standard walk matrix determines the full adjacency spectrum and is compatible with at most two adjacency matrices. Their $8$-vertex example produces two distinct adjacency matrices but isomorphic graphs, which led them to conjecture that this is always so~\cite[Remark~5.2]{LiuSiemons2022}. This conjecture was subsequently disproved by Lv, Liu, Wang and Wang \cite{LvLiuWangWang2026}, who exhibited two non-isomorphic graphs on $10$ vertices with the same walk matrix of rank $8$.

The neighbouring rank-$(n-1)$ case has also been studied from the viewpoint of generalized spectral characterization \cite{WangLiuWang2021}. Liu and Siemons also gave an
operator-theoretic characterization of walk equivalence: for standard walk matrices, \(W_G=W_H\) if and only if the restrictions of $A_G$ and $A_H$ to the common walk space coincide \cite[Proposition~6.1]{LiuSiemons2022}. The order-$10$ example of Lv et al. settles the isomorphism question negatively, while the operator-theoretic description above leaves open a concrete graph-theoretic question: what exactly can the second adjacency matrix look like at corank two? We answer this question completely. Every such ambiguity arises from a reciprocal form of WQH switching. As consequences, we prove that order $10$ is minimal and construct connected non-isomorphic examples in every order $n\ge10$.

Switching is a classical source of cospectral graphs; Godsil--McKay switching is the standard example \cite{GodsilMcKay1982}. Regular orthogonal matrices place several such constructions in a common linear-algebraic framework \cite{AbiadHaemers2012,MaoWangLiuQiu2023,AbiadVanDeBergSimoens2026}. Of particular relevance is the switching of Wang, Qiu and Hu \cite{WangQiuHu2019}. If $B$ is obtained from $A$ by WQH switching, then \(B=Q^TAQ, Q\mathbf e=\mathbf e,\) for a regular orthogonal matrix $Q$, and consequently \(W(B)=Q^TW(A).\) Thus WQH switching preserves the adjacency spectrum. Exact preservation of the walk matrix is stricter: $Q$ must fix the entire Krylov space pointwise. 

The isomorphism problem for WQH switching has recently been studied by Abiad and Ge~\cite{AbiadGe2026}, who developed common-neighbour criteria for detecting when a WQH-switched graph is non-isomorphic to the original one and obtained infinite families of such pairs. Their setting concerns adjacency cospectrality; the exact walk-matrix condition considered here is substantially more rigid.

We call a WQH partition $(C_1,C_2;D_1,D_2;R)$ \emph{reciprocal} if, after exchanging the two switching sides, $(D_1,D_2;C_1,C_2;R)$ is again a WQH partition. The precise definition is given in Section~\ref{sec:preliminaries}.

\begin{theorem}[Corank-two characterization]\label{thm:characterization}
Let $G$ be a graph of order $n$ with \(\operatorname{rank}W_G=n-2,\) and let $H\ne G$ be a graph on the same labelled vertex set. Then \(W_H=W_G\) if and only if $H$ is obtained from $G$ by a reciprocal WQH switching. Moreover, if $A=A_G$ and $B=A_H$, there exist nonzero, disjointly supported vectors $u,v\in\{0,\pm1\}^n$ such that \(\ker W_G^T=\operatorname{span}\{u,v\}, A-B=uv^T+vu^T.\) Writing \(q=\frac12u^Tu, s=\frac12v^Tv,\) one has
\[
B=\left(I-\frac1q uu^T\right)A\left(I-\frac1q uu^T\right)
=\left(I-\frac1s vv^T\right)A\left(I-\frac1s vv^T\right).
\]
\end{theorem}

So the ambiguity is not merely two-dimensional in a numerical sense: the two-dimensional space $\ker W_G^T$ actually carries the whole switching mechanism. Conversely, reciprocal WQH switching preserves every walk vector, even when the walk matrix has larger corank.

The isomorphism question then becomes sharp. If one switching pair has size one, the corresponding reflection is simply a transposition, hence gives an isomorphic graph. Orders $8$ and $9$ still cannot support a non-isomorphic corank-two pair. Order $10$ can, and from that seed every larger order can be reached.

\begin{theorem}[Order dichotomy]\label{thm:order-dichotomy}
There exist non-isomorphic graphs $G,H$ of order $n$ satisfying \(W_G=W_H, \operatorname{rank}W_G=n-2\) if and only if \(n\ge10.\) Moreover, for every $n\ge10$, such a pair may be chosen connected.
\end{theorem}

The $10$-vertex existence part was already witnessed by the counterexample of Lv et al.~\cite{LvLiuWangWang2026}. The new point here is that $10$ is the minimum possible order, and that connected examples exist in every order \(n\ge10\).

The proof of the lower bound is structural. For \(n=8\), the balance relations place the Krylov space in a two-dimensional $A$-invariant subspace; for $n=9$, a slightly
finer argument places it in an invariant subspace of dimension at most $5$. For the positive direction, we fix a convenient labelled realization $G_{10},H_{10}$ of the pair of Lv et al., expose its reciprocal WQH structure, and use it as the seed of the infinite construction. We then use
\[
\mathcal U(X):=X\vee K_1,\quad \mathcal T(X):=(X\cup K_1)\vee K_1.
\]
Singleton union/join operations have appeared before in walk-matrix constructions \cite{LiuSiemonsWang2019}. Here the reciprocal WQH core survives both operations, while the walk rank is tracked by the coronal
\[
\Gamma_X(x)=\mathbf e^T(xI-A_X)^{-1}\mathbf e.
\]
A short non-cancellation argument shows that $\mathcal T$ raises both order and walk rank by two; one initial application of $\mathcal U$ supplies the odd orders.

\section{Preliminaries}\label{sec:preliminaries}

For a graph $G$ with adjacency matrix $A$, write
\[
\mathcal K_G:=\operatorname{col}W_G.
\]

\begin{lemma}\label{lem:krylov}
Let $r=\operatorname{rank}W_G$. Then
\[
\mathcal K_G=\operatorname{span}\{\mathbf e,A\mathbf e,\ldots,A^{r-1}\mathbf e\}
\]
is $A$-invariant. Since $A$ is symmetric, $\mathcal K_G^\perp=\ker W_G^T$ is also $A$-invariant. In particular, if $r=n-2$, then \(\dim\ker W_G^T=2.\)
\end{lemma}

\begin{proof}
Consider the infinite Krylov sequence
\[
\mathbf e,A\mathbf e,A^2\mathbf e,\ldots .
\]
By the Cayley--Hamilton theorem, there exists a least integer \(m\leq n\) such that \(A^m\mathbf e \in \operatorname{span}\{\mathbf e,A\mathbf e,\ldots,A^{m-1}\mathbf e\}.\)
By the minimality of $m$, the vectors \(\mathbf e,A\mathbf e,\ldots,A^{m-1}\mathbf e\) are linearly independent. Multiplying the above relation successively by $A$ shows that every later Krylov vector lies in the same span. Hence \(m=r\), and therefore \(\mathcal K_G=\operatorname{span}\{\mathbf e,A\mathbf e,\ldots,A^{r-1}\mathbf e\}\)
is $A$-invariant.

Since $A$ is symmetric, \(\mathcal K_G^\perp\) is also $A$-invariant. Indeed, if \(x\in\mathcal K_G^\perp\) and \(y\in\mathcal K_G\), then \(Ay\in\mathcal K_G\), and hence
\(\langle Ax,y\rangle=\langle x,Ay\rangle=0.\) Thus \(Ax\in\mathcal K_G^\perp\). Finally, \(\mathcal K_G^\perp=\ker W_G^T,\) so if \(r=n-2\), then \(\dim\ker W_G^T=2.\)
\end{proof}

We next recall WQH switching in the form convenient here. For $X\subseteq V(G)$, let $\mathbf1_X$ denote its characteristic vector, and for $z\in V(G)$ let $\mathbf e_z$ denote the corresponding standard basis vector. We write $d_X(z):=|N_G(z)\cap X|$.

\begin{definition}\label{def:WQH}
Let
\[
V(G)=C_1\sqcup C_2\sqcup D_1\sqcup D_2\sqcup R,
\quad |C_1|=|C_2|=:q>0.
\]
This is a \emph{WQH partition} if there is an integer $\alpha$ such that
\[
d_{C_1}(x)-d_{C_2}(x)=\alpha\quad(x\in C_1),
\quad
d_{C_2}(x)-d_{C_1}(x)=\alpha\quad(x\in C_2),
\]
every vertex of $C_1$ is adjacent to all vertices of $D_1$ and to none of $D_2$, every vertex of $C_2$ is adjacent to all vertices of $D_2$ and to none of $D_1$, and
\(d_{C_1}(z)=d_{C_2}(z)(z\in R).\) The associated WQH switching interchanges the complete blocks $C_1-D_1,C_2-D_2$ with $C_1-D_2,C_2-D_1$, leaving all other adjacencies unchanged.
\end{definition}

\begin{definition}\label{def:reciprocal}
A WQH partition $(C_1,C_2;D_1,D_2;R)$ is \emph{reciprocal} if $(D_1,D_2;C_1,C_2;R)$ is again a WQH partition. In particular, \(|D_1|=|D_2|=s>0.\)
\end{definition}

The two orientations switch exactly the same adjacencies between \(C_1\cup C_2\) and \(D_1\cup D_2\), and hence produce the same switched graph. We call this graph the \emph{reciprocal WQH mate} of $G$.

The following matrix identities will be used repeatedly.

\begin{proposition}\label{prop:WQH-basic}
Let $(C_1,C_2;D_1,D_2;R)$ be a WQH partition of $G$, and put
\(
u=\mathbf1_{C_1}-\mathbf1_{C_2},
v=\mathbf1_{D_1}-\mathbf1_{D_2}.
\)
If $A$ is the adjacency matrix of $G$, then \(Au=\alpha u+qv.\) Moreover, \(Q_u:=I-\frac1q uu^T\) is orthogonal, satisfies $Q_u\mathbf e=\mathbf e$, and the switched adjacency matrix $B$ is
\[
B=Q_uAQ_u=A-uv^T-vu^T.
\]
Consequently,
\(
W(B)=Q_uW(A),
W(B)=W(A)\Longleftrightarrow u\in\ker W(A)^T.
\)
\end{proposition}

\begin{proof}
Read $Au$ on the five parts $C_1,C_2,D_1,D_2,R$; this gives $Au=\alpha u+qv$. Since $u^Tu=2q$ and $u^T\mathbf e=0$, we have $Q_u^2=I$ and $Q_u\mathbf e=\mathbf e$. Also $u^TA=\alpha u^T+qv^T$, so a direct expansion gives
\[
Q_uAQ_u=A-uv^T-vu^T,
\]
which is precisely the WQH-switched adjacency matrix. Since $Q_u^2=I$ and $Q_u\mathbf e=\mathbf e$, we have
\[
B^k=(Q_uAQ_u)^k=Q_uA^kQ_u
\]
for every $k\ge0$. Hence $B^k\mathbf e=Q_uA^k\mathbf e$. Therefore \(W(B)=Q_uW(A).\) Moreover, $Q_u$ fixes a vector $x$ if and only if $x\perp u$. Thus $W(B)=W(A)$ if and only if $u$ is orthogonal to every column of $W(A)$, that is, \(u\in\ker W(A)^T.\)
\end{proof}

Reciprocity gives one further fact, crucial for the converse direction of Theorem~\ref{thm:characterization}.

\begin{proposition}\label{prop:reciprocal-walks}
Suppose $(C_1,C_2;D_1,D_2;R)$ is reciprocal, and let $H$ be the reciprocal WQH mate of $G$. Then
\[
A_G^k\mathbf e=A_H^k\mathbf e \quad (k\ge0).
\]
In particular,
\(
W_G=W_H,
\operatorname{rank}W_G\le n-2.
\)
\end{proposition}

\begin{proof}
Write \(A=A_G, B=A_H.\) Let \(|D_1|=|D_2|=s\), and set
\(
u=\mathbf1_{C_1}-\mathbf1_{C_2},
v=\mathbf1_{D_1}-\mathbf1_{D_2}.
\)
Applying Proposition~\ref{prop:WQH-basic} from the two sides gives \(Au=\alpha u+qv, Av=su+\delta v\) for some integers $\alpha,\delta$. Hence \(L:=\operatorname{span}\{u,v\}\) is $A$-invariant. Since $A$ is symmetric and $e\perp L$, we have
\[
A^k\mathbf e\in L^\perp \quad (k\ge0).
\]
By Proposition~\ref{prop:WQH-basic}, \(B=A-uv^T-vu^T.\) Hence, for every \(x\in L^\perp\),
\((A-B)x
=(uv^T+vu^T)x=0.\)
It follows inductively that
\(
B^ke=A^ke (k\ge0).
\)
Therefore \(W_G=W_H.\)

Since $u$ and $v$ are nonzero and have disjoint supports, they are linearly independent, and hence $\dim L=2$. Finally, \(\mathcal K_G\subseteq L^\perp, \dim L^\perp=n-2,\)
and hence \(\operatorname{rank}W_G\le n-2.\)
\end{proof}

\section{Corank-Two Walk Equivalence}\label{sec:characterization}

The codimension-two assumption leaves very little room for the difference of two compatible adjacency matrices.

\begin{lemma}\label{lem:difference-rank-two}
Assume $G\ne H$, \(W_G=W_H=W, \operatorname{rank}W=n-2,\) and put $D=A_G-A_H$. Then
\(
\operatorname{rank}D=2,
\ker D=\operatorname{col}W,
\operatorname{Im}D=\ker W^T.
\)
\end{lemma}

\begin{proof}
By Lemma~\ref{lem:krylov},
\[
\operatorname{col}W=\operatorname{span}\{\mathbf e,A_G\mathbf e,\ldots,A_G^{n-3}\mathbf e\}.
\]
For $0\le k\le n-3$, equality of the walk matrices gives $A_G^k\mathbf e=A_H^k\mathbf e$, and therefore
\[
DA_G^k\mathbf e
=A_G^{k+1}\mathbf e-A_HA_G^k\mathbf e
=A_G^{k+1}\mathbf e-A_H^{k+1}\mathbf e
=0.
\]
Hence $\operatorname{col}W\subseteq\ker D$, so $\operatorname{rank}D\le2$. Now $D\ne0$ because $G\ne H$. Nor can $D$ have rank one. Indeed, if \(D=\lambda xx^T\) for some 
\(\lambda\neq0\) and \(x\neq0\), then \(D_{ii}=\lambda x_i^2\) for every $i$. Since $D$ has zero diagonal, this would force \(x_i=0\) for every  $i$, a contradiction. Hence \(\operatorname{rank}D=2.\) Since \(\operatorname{col}W\subseteq\ker D\) and both spaces have dimension $n-2$, we obtain \(\ker D=\operatorname{col}W.\) Since $D$ is symmetric, \(\operatorname{Im}D=(\ker D)^\perp=\ker W^T.\)
\end{proof}

\begin{lemma}\label{lem:rank-two-sign}
Let $D\ne0$ be symmetric, of rank $2$, with zero diagonal and entries in $\{0,\pm1\}$. Then there are nonzero, disjointly supported vectors \(u,v\in\{0,\pm1\}^n\) such that
\(D=uv^T+vu^T.\)
\end{lemma}

\begin{proof}
Since $\operatorname{tr}D=0$, the two nonzero eigenvalues of $D$ are $\lambda$ and $-\lambda$. Thus \(D=\lambda(xx^T-yy^T)\) for orthonormal vectors $x,y$.

The condition $D_{ii}=0$ gives \(|x_i|=|y_i|\) for every $i$. Hence, on their common support,
\[
y_i=\varepsilon_i x_i, \quad
\varepsilon_i\in\{\pm1\}.
\]
Since $x\perp y$,
\[
0=x^Ty=\sum_i \varepsilon_i x_i^2.
\]
On the common support of $x$ and $y$, all $x_i^2$ are positive. Hence both signs $+1$ and $-1$ must occur among the \(\varepsilon_i\). Put \(P=\{i:\varepsilon_i=1\}, N=\{i:\varepsilon_i=-1\}.\) Then \(D_{ij}=\lambda x_ix_j(1-\varepsilon_i\varepsilon_j),\) so $D_{ij}=0$ when $i,j$ lie in the same one of $P,N$, and $D_{ij}\neq0$ when $i\in P$, $j\in N$.

Thus, after deleting the zero rows and columns,
\[
D=
\begin{pmatrix}
0&M\\
M^T&0
\end{pmatrix}.
\]
The kernel of this block matrix is \(\ker M^T\oplus\ker M,\) and hence
\[
\operatorname{rank}
\begin{pmatrix}
0&M\\
M^T&0
\end{pmatrix}
=2\operatorname{rank}M.
\]
Since \(\operatorname{rank}D=2\), it follows that \(\operatorname{rank}M=1.\) Every entry of $M$ belongs to $\{\pm1\}$. Choose \(i_0\in P\) and \(j_0\in N\). Since every \(2\times2\) minor of \(M\) vanishes,
\[
M_{ij}M_{i_0j_0}=M_{ij_0}M_{i_0j},
\]
and hence \(M=ab^T\) for suitable \(a\in\{\pm1\}^{P}\) and \(b\in\{\pm1\}^{N}\). Extending $a$ and $b$ by zero outside $P$ and $N$, respectively, gives disjointly supported vectors \(u,v\in\{0,\pm1\}^n\) such that \(D=uv^T+vu^T.\)
\end{proof}

\begin{proof}[Proof of Theorem~\ref{thm:characterization}]
Assume first that $W_G=W_H$, and write $A=A_G$, $B=A_H$. Lemmas~\ref{lem:difference-rank-two} and \ref{lem:rank-two-sign} give \(A-B=uv^T+vu^T\) for nonzero, disjointly supported $u,v\in\{0,\pm1\}^n$. Since \(W_G=W_H\), their second columns agree, and therefore \(A\mathbf e=B\mathbf e.\) Thus
\[
\mathbf e\in\ker(A-B).
\]
Consequently,
\[
0=(A-B)\mathbf e
=(v^T\mathbf e)u+(u^T\mathbf e)v.
\]

The vectors $u,v$ are linearly independent, so \(u^T\mathbf e=v^T\mathbf e=0.\) Define
\[
C_1=\{i:u_i=1\},\quad C_2=\{i:u_i=-1\},\quad
D_1=\{i:v_i=1\},\quad D_2=\{i:v_i=-1\},
\]
and let $R$ consist of the remaining vertices. Put
\(
C:=C_1\cup C_2,
D:=D_1\cup D_2.\)
Then
\[
|C_1|=|C_2|=q,\quad |D_1|=|D_2|=s,
\]
with $q,s\ge1$.

The identity \(A-B=uv^T+vu^T\) already determines where the two graphs differ. Indeed, for \(i\in C\) and \(j\in D\),
\[
(A-B)_{ij}=u_i v_j.
\]
Hence \((A-B)_{ij}=1\) on \(C_1\times D_1\) and \(C_2\times D_2\), while \((A-B)_{ij}=-1\) on \(C_1\times D_2\) and \(C_2\times D_1\). Since \(A_{ij},B_{ij}\in\{0,1\}\), the \(C\times D\) blocks are therefore
\[
\begin{array}{c|cc}
      &D_1&D_2\\ \hline
C_1&J&0\\
C_2&0&J
\end{array}
\quad\text{in }A,
\qquad
\begin{array}{c|cc}
      &D_1&D_2\\ \hline
C_1&0&J\\
C_2&J&0
\end{array}
\quad\text{in }B,
\]
where \(J\) denotes an all-one block of the appropriate size. Outside \(C\times D\) and its transpose, \(uv^T+vu^T\) vanishes, so \(A\) and \(B\) agree.

By Lemma~\ref{lem:difference-rank-two},
\[
\ker W^T=\operatorname{Im}(A-B)=\operatorname{span}\{u,v\}.
\]
Lemma~\ref{lem:krylov} makes this plane $A$-invariant. Hence \(Au=\alpha u+\beta v, Av=\gamma u+\delta v\) for real $\alpha,\beta,\gamma,\delta$. Take $x\in D_1$. The checkerboard pattern gives $(Au)_x=q$, while $u_x=0$ and $v_x=1$; therefore $\beta=q$. Similarly, evaluating $Av$ at any vertex of $C_1$ gives $\gamma=s$. Thus
\[
Au=\alpha u+qv,\quad Av=su+\delta v.
\]
Since $C_1,D_1$ are nonempty, the same coordinate evaluations show that $\alpha,\delta\in\mathbb Z$.

Now read these two identities coordinatewise. The first yields
\[
d_{C_1}(x)-d_{C_2}(x)=\alpha\quad(x\in C_1),
\quad
d_{C_2}(x)-d_{C_1}(x)=\alpha\quad(x\in C_2),
\]
and $d_{C_1}(x)=d_{C_2}(x)$ for $x\in R$. Together with the checkerboard block, these are precisely the WQH conditions for $(C_1,C_2;D_1,D_2;R)$. The second identity gives the same conclusion with $C$ and $D$ interchanged. Hence the partition is reciprocal, and $H$ is exactly the corresponding reciprocal WQH mate of $G$.

Conversely, suppose that $H$ is obtained from $G$ by reciprocal WQH switching. Let
\[
u=\mathbf 1_{C_1}-\mathbf 1_{C_2},
\quad
v=\mathbf 1_{D_1}-\mathbf 1_{D_2}.
\]
By Proposition~\ref{prop:reciprocal-walks}, $W_H=W_G$ and \(\mathcal K_G\subseteq \operatorname{span}\{u,v\}^{\perp}.\) Hence \(\operatorname{span}\{u,v\}\subseteq\ker W_G^T.\) The vectors $u$ and $v$ are nonzero and have disjoint supports, so they are linearly independent. Since \(\operatorname{rank}W_G=n-2,\) we have \(\dim\ker W_G^T=2,\)
and therefore \(\ker W_G^T=\operatorname{span}\{u,v\}.\) Moreover, Proposition~\ref{prop:WQH-basic} gives \(A-B=uv^T+vu^T.\)

Finally, applying Proposition~\ref{prop:WQH-basic} from the two switching sides gives
\[
B=\left(I-\frac1q uu^T\right)
A
\left(I-\frac1q uu^T\right)
=
\left(I-\frac1s vv^T\right)
A
\left(I-\frac1s vv^T\right).
\]
\end{proof}

\begin{remark}\label{rem:two-reflections}
The ambiguity is confined to \(\ker W_G^T=\mathcal K_G^\perp,\) the two-dimensional orthogonal complement of the Krylov space. The same mate is obtained from the two reflections above, whose reflecting directions span this space.
\end{remark}

\section{The Minimum Order}\label{sec:minorder}

The lower-bound argument reduces the cases $n=8,9$ to two elementary facts about graphs on four vertices.

\begin{lemma}\label{lem:four-vertex}
Let $X$ be a graph on four vertices, partitioned as
\[
V(X)=X_1\sqcup X_2,
\quad
|X_1|=|X_2|=2.
\]
Then the following hold.
\begin{enumerate}
\item If \(w=\mathbf 1_{X_1}-\mathbf 1_{X_2}\) is an eigenvector of $A(X)$, then $X$ is regular.
\item If $X$ is regular and $\xi\in\{0,1\}^4$ has even weight, then \(A(X)\xi \in \operatorname{span}\{\mathbf 1_{V(X)},\xi\}.\)
\end{enumerate}
\end{lemma}

\begin{proof}
For (1), let $a,b\in\{0,1\}$ indicate whether the two vertices inside $X_1$ and $X_2$, respectively, are adjacent. If \(A(X)w=\lambda w,\) then every vertex of $X_1$ has $a-\lambda$ neighbours in $X_2$, while every vertex of $X_2$ has $b-\lambda$ neighbours in $X_1$. Double counting the edges between $X_1$ and $X_2$ gives $a=b$. Hence every vertex has degree \(a+(a-\lambda)=2a-\lambda,\) so $X$ is regular.

For (2), the cases in which $\xi$ has weight $0$ or $4$ are immediate, since $X$ is regular. It remains to consider the case in which $\xi$ has weight $2$. A regular graph on four vertices is isomorphic to one of \(4K_1, 2K_2, C_4, K_4.\) If $X=4K_1$, then \(A(X)\xi=0.\) If $X=2K_2$, then either the two vertices in $\operatorname{supp}(\xi)$ form an edge, in which case \(A(X)\xi=\xi,\) or they lie in different components, in which case \(A(X)\xi=\mathbf 1_{V(X)}-\xi.\) If $X=C_4$, then for two adjacent vertices in
$\operatorname{supp}(\xi)$, \(A(X)\xi=\mathbf 1_{V(X)},\) whereas for two opposite vertices, \(A(X)\xi=2\bigl(\mathbf 1_{V(X)}-\xi\bigr).\) Finally, if $X=K_4$, then
\(A(X)\xi=2\mathbf 1_{V(X)}-\xi.\) Thus in every case \(A(X)\xi \in \operatorname{span}\{\mathbf 1_{V(X)},\xi\}.\)
\end{proof}

\begin{proposition}\label{prop:no-small-order}
If $G,H$ are non-isomorphic and satisfy \(W_G=W_H, \operatorname{rank}W_G=n-2,\) then $n\ge10$.
\end{proposition}

\begin{proof}
Let \(A=A_G\), and choose $u,v$ and the reciprocal WQH partition as in Theorem~\ref{thm:characterization}. Put \(q=|C_1|=|C_2|, s=|D_1|=|D_2|.\) If $q=1$, then \(Q_u=I-uu^T\) is simply the permutation matrix interchanging the two vertices of $C_1\cup C_2$. Hence $G\cong H$. The same argument applies if $s=1$. Thus a non-isomorphic pair must have $q,s\ge2$, and therefore $n\ge8$.

\emph{Case $n=8$.} Here $q=s=2$ and $R=\varnothing$. Put $C=C_1\cup C_2$ and $D=D_1\cup D_2$. From \(Au=\alpha u+2v, Av=2u+\delta v\) and Lemma~\ref{lem:four-vertex}(1), both $G[C]$ and $G[D]$ are regular, say of degrees $k_C,k_D$. The checkerboard block contributes two neighbours across the partition to every vertex. Hence
\[
A\mathbf1_C=k_C\mathbf1_C+2\mathbf1_D,
\quad
A\mathbf1_D=2\mathbf1_C+k_D\mathbf1_D.
\]
Thus $\operatorname{span}\{\mathbf1_C,\mathbf1_D\}$ is $A$-invariant and contains $\mathbf e$. Therefore $\operatorname{rank}W_G\le2$, contradicting $\operatorname{rank}W_G=6$.

\emph{Case $n=9$.} Again $q=s=2$, but now $R=\{z\}$. The induced graphs $G[C]$ and $G[D]$ are regular. Set \(\xi=\mathbf 1_{N(z)\cap C}, \eta=\mathbf 1_{N(z)\cap D}.\)
Let $A_C$ and $A_D$ denote the adjacency matrices of $G[C]$ and $G[D]$, respectively. By a slight abuse of notation, we also use $A_C$ and $A_D$ for their zero extensions to $\mathbb R^{V(G)}$. By reciprocal balance,
\[
|N(z)\cap C_1|=|N(z)\cap C_2|=a,
\quad
|N(z)\cap D_1|=|N(z)\cap D_2|=b.
\]
Hence $\xi$ and $\eta$ have weights $2a$ and $2b$, respectively, and therefore both have even weight. Therefore Lemma~\ref{lem:four-vertex}(2) gives
\[
A_C\xi\in\operatorname{span}\{\mathbf1_C,\xi\},
\quad
A_D\eta\in\operatorname{span}\{\mathbf1_D,\eta\}.
\]
Then
\[
\begin{aligned}
A\mathbf1_C&=k_C\mathbf1_C+2\mathbf1_D+2a\mathbf e_z,\\
A\mathbf1_D&=2\mathbf1_C+k_D\mathbf1_D+2b\mathbf e_z,\\
A\mathbf e_z&=\xi+\eta,\\
A\xi&=A_C\xi+a\mathbf1_D+2a\mathbf e_z,\\
A\eta&=b\mathbf1_C+A_D\eta+2b\mathbf e_z.
\end{aligned}
\]
Consequently \(\mathcal L=\operatorname{span}\{\mathbf1_C,\mathbf1_D,\mathbf e_z,\xi,\eta\}\) is $A$-invariant and contains $\mathbf e$. Thus $\operatorname{rank}W_G\le5$, again impossible because the required rank is $7$.
\end{proof}

A non-isomorphic pair of order $10$ with the same walk matrix of rank $8$ was previously exhibited by Lv et al. \cite[Fig.~4]{LvLiuWangWang2026}. For later use in the infinite-family construction, we fix a convenient labelled realization of their pair. Let 
\[
C_1=\{1,2\},\quad C_2=\{3,4\},\quad
D_1=\{5,6\},\quad D_2=\{7,8\},\quad R=\{9,10\},
\]
and set
\[
F=\{(1,4),(1,9),(1,10),(2,3),(3,9),(4,10), (5,9),(5,10),(7,10),(8,9)\}.
\]
Here $(i,j)$ means the unordered edge $\{i,j\}$, and $C_i\times D_j$ denotes the complete bipartite edge set between $C_i$ and $D_j$. Define
\[
E(G_{10})=F\cup(C_1\times D_1)\cup(C_2\times D_2),
\quad
E(H_{10})=F\cup(C_1\times D_2)\cup(C_2\times D_1).
\]
Thus the two graphs are completely specified by the common edge set $F$ together with the two complementary checkerboard blocks above.

The present labelling is chosen to make the reciprocal WQH structure transparent.

\begin{proposition}\label{prop:order-ten}
For the labelled realization above, $G_{10}$ and $H_{10}$ are reciprocal WQH mates. Moreover, \(W_{G_{10}}=W_{H_{10}}, \operatorname{rank}W_{G_{10}}=8,\) and the two graphs are connected and non-isomorphic.
\end{proposition}

\begin{proof}
Put
\[
u=(1,1,-1,-1,0,0,0,0,0,0)^T,
\quad
v=(0,0,0,0,1,1,-1,-1,0,0)^T.
\]
Let \(A=A(G_{10})\). A direct calculation gives
\(
Au=-u+2v, Av=2u.\)
Together with the checkerboard form built into the definition of \(G_{10}\), these identities give the WQH balance conditions from both sides. Hence the displayed partition is reciprocal. Proposition~\ref{prop:reciprocal-walks} gives $W_{G_{10}}=W_{H_{10}}$ and rank at most $8$.

Both graphs are connected. Indeed, the checkerboard part gives two connected bipartite pieces on $C_1\cup D_1$ and $C_2\cup D_2$ (or the crossed pair in $H_{10}$), the edge $(1,4)$ joins them, and each of $9,10$ has a neighbour in this core. On the other hand, the minor of
\[
[e,Ae,\ldots,A^7e]
\]
on rows
\[
\{1,2,3,5,6,7,9,10\}
\]
has determinant $48\neq0$; see Appendix~\ref{app:seed-verification} for the explicit matrix. Hence $\operatorname{rank}W_{G_{10}}\ge8$. Together with the upper bound from Proposition~\ref{prop:reciprocal-walks}, we obtain $\operatorname{rank}W_{G_{10}}=8$.

Finally, in the displayed labelling both graphs have degree list $(5,3,4,4,4,2,3,3,4,4)$. Vertex $6$ is the unique vertex of degree $2$. Its neighbours have degrees $\{5,3\}$ in $G_{10}$ and $\{4,4\}$ in $H_{10}$. No isomorphism can preserve these local degree data. Hence $G_{10}\not\cong H_{10}$.
\end{proof}

\begin{corollary}\label{cor:min-order}
The smallest order of a pair of non-isomorphic graphs $G,H$ satisfying \(W_G=W_H, \operatorname{rank}W_G=n-2\) is $10$.
\end{corollary}

\begin{proof}
Proposition~\ref{prop:no-small-order} rules out all smaller orders, while the \(10\)-vertex example of Lv et al. \cite{LvLiuWangWang2026}, equivalently the labelled realization above, attains the bound.
\end{proof}

\section{Infinite Families}\label{sec:families}
For a graph $X$, define
\[
\mathcal U(X)=X\vee K_1,
\quad
\mathcal T(X)=(X\cup K_1)\vee K_1.
\]
Here $\vee$ denotes the graph join and $\cup$ the disjoint union. Thus $\mathcal U$ adds a universal vertex, while $\mathcal T$ first adds an isolated vertex and then a universal one. Singleton union/join operations have proved useful in related walk-matrix constructions \cite{LiuSiemonsWang2019}.

\begin{lemma}\label{lem:UT-reciprocal}
If $G,H$ are reciprocal WQH mates, then so are the pairs
\[
\mathcal U(G),\mathcal U(H)
\quad\text{and}\quad
\mathcal T(G),\mathcal T(H).
\]
\end{lemma}

\begin{proof}
Keep $C_1,C_2,D_1,D_2$ unchanged and put the new vertices into the balanced class $R$. The universal vertex has equally many neighbours in $C_1,C_2$ and in $D_1,D_2$, while the vertex introduced by the disjoint union has no neighbours in $C_1\cup C_2\cup D_1\cup D_2$. Hence all balance conditions remain valid, and reciprocity is preserved.
\end{proof}

To control the walk rank, we use the graph coronal introduced by McLeman and McNicholas~\cite{McLemanMcNicholas2011},
\[
\Gamma_X(x)=\mathbf e^T(xI-A_X)^{-1}\mathbf e.
\]
Write the coronal in reduced form as
\[
\Gamma_X(x)=\frac{h(x)}{m(x)},
\]
with $m$ monic. By the spectral decomposition,
\[
\Gamma_X(x)=
\sum_{\lambda\ {\rm main}}
\frac{\|E_\lambda \mathbf e\|^2}{x-\lambda},
\]
so its poles are precisely the distinct main eigenvalues of $X$. Thus
\[
m(x)=\prod_{\lambda\ {\rm main}}(x-\lambda)
\]
is the main polynomial of $X$. Hence, by the standard relation between main eigenvalues and the walk matrix~\cite{Hagos2002,LiuSiemons2022}, \(\deg m=\operatorname{rank}W_X.\)

\begin{proposition}\label{prop:rank-lifting}
Let \(\Gamma_X(x)=\frac{h(x)}{m(x)}\) be reduced, with $m$ monic. Write the reduced coronals as
\[
\Gamma_{U(X)}(x)=\frac{h_U(x)}{m_U(x)},
\quad
\Gamma_{T(X)}(x)=\frac{h_T(x)}{m_T(x)},
\]
where $m_U$ and $m_T$ are monic. Then the following hold.

\begin{enumerate}
\item \[\Gamma_{U(X)}(x) = \frac{(x+2)h(x)+m(x)}{xm(x)-h(x)}.\]
If $m(-1)+h(-1)\neq0$, then \(\operatorname{rank}W_{U(X)}= \operatorname{rank}W_X+1,\) and \(m_U(0)=-h(0), h_U(-1)=m(-1)+h(-1).\)

\item \[\Gamma_{T(X)}(x)=\frac{x(x+2)h(x)+2(x+1)m(x)}{(x^2-1)m(x)-xh(x)}.\]
If $m(0)h(-1)\neq0$, then \(\operatorname{rank}W_{T(X)}=\operatorname{rank}W_X+2,\) and \(m_T(0)=-m(0), h_T(-1)=-h(-1).\) In particular, the two non-vanishing conditions \(m_T(0)\neq0, h_T(-1)\neq0\) persist under iteration of $T$.
\end{enumerate}
\end{proposition}

\begin{proof}
Since \(\Gamma_X\) is proper, \(\deg h<\deg m\).
Write
\[
\gamma(x):=\Gamma_X(x)
=\mathbf e^T(xI-A_X)^{-1}\mathbf e.
\]
The adjacency matrix of \(\mathcal U(X)\) is
\[
\begin{pmatrix}
  A_X & \mathbf e\\
    \mathbf e^T & 0
\end{pmatrix}.
\]
Set \(\mathcal R(x):=(xI-A_X)^{-1}.\) To compute the coronal of \(\mathcal U(X)\), solve
\[
\begin{pmatrix}
 xI-A_X & -\mathbf e\\
  -\mathbf e^T & x
\end{pmatrix}
\binom{y}{z}
 = \binom{\mathbf e}{1}.
\]
The first block equation gives \(y=(1+z)\mathcal R(x)\mathbf e.\) Substituting this into the second equation yields \(z=\frac{1+\gamma}{x-\gamma}.\) Consequently,
\[
\Gamma_{\mathcal U(X)}(x)
=\mathbf e^Ty+z
=\gamma+(1+\gamma)z
=\frac{1+(x+2)\gamma}{x-\gamma}.
\]
Since \(\gamma=h/m\), we obtain
\[
\Gamma_{\mathcal U(X)}(x)=
\frac{(x+2)h+m}{xm-h}.
\]

Set \(D_U=xm-h, N_U=(x+2)h+m.\) Then \((x+2)D_U+N_U=(x+1)^2m.\) If $\lambda$ is a common zero of $D_U$ and $N_U$, the preceding identity gives \((\lambda+1)^2m(\lambda)=0.\)
If $m(\lambda)=0$, then, because $\gcd(h,m)=1$,
\(
D_U(\lambda)
=\lambda m(\lambda)-h(\lambda)
=-h(\lambda)\neq0,
\)
a contradiction. Hence necessarily $\lambda=-1$.

At this point,
\[
D_U(-1)=-m(-1)-h(-1),
\quad
N_U(-1)=m(-1)+h(-1).
\]
Thus the condition $m(-1)+h(-1)\neq0$ rules out every possible cancellation. The denominator is then monic of degree $\deg m+1$, proving the rank increase. The evaluations at $0$ and $-1$ give the recurrence formulas.

For \(\mathcal T\), put \(Y=X\cup K_1\). Since the coronal is additive under disjoint union,
\[
\Gamma_Y(x)=
\Gamma_X(x)+\frac1x=\frac{xh+m}{xm}.
\]
Now \(\mathcal T(X)=\mathcal U(Y).\) Applying the formula just proved to \(Y\) gives
\[
\Gamma_{\mathcal T(X)}(x)=
\frac{1+(x+2)\Gamma_Y(x)}{x-\Gamma_Y(x)}=
\frac{x(x+2)h+2(x+1)m}{(x^2-1)m-xh}.
\]

Put
\(
D_T(x):=(x^2-1)m(x)-xh(x),
N_T(x):=x(x+2)h(x)+2(x+1)m(x).\)
Then
\[
(x+2)D_T(x)+N_T(x)=x(x+1)^2m(x).
\]

Suppose that $\lambda$ is a common zero of $D_T$ and $N_T$. The preceding identity gives \(\lambda(\lambda+1)^2m(\lambda)=0.\) If $m(\lambda)=0$, then, since $\gcd(h,m)=1$, we have $h(\lambda)\neq0$. But \(D_T(\lambda)=(\lambda^2-1)m(\lambda)-\lambda h(\lambda)=-\lambda h(\lambda).\) Thus, if $\lambda\neq0$, this contradicts $D_T(\lambda)=0$. If $\lambda=0$, then this case is already among the exceptional values singled out by the factor $\lambda(\lambda+1)^2$.

Consequently, any cancellation can occur only at $\lambda=0$ or $\lambda=-1$. At these two points, \(D_T(0)=-m(0), N_T(0)=2m(0),\) and \(D_T(-1)=h(-1), N_T(-1)=-h(-1).\)
Hence the condition $m(0)h(-1)\neq0$ rules out all possible cancellation.

Therefore the displayed fraction for $\Gamma_{T(X)}$ is already reduced. Since $m$ is monic and $\deg h<\deg m$, the denominator $D_T$ is monic of degree $\deg m+2$.

It follows that
\[
\operatorname{rank}W_{T(X)}=\deg m+2
=\operatorname{rank}W_X+2.
\]
Moreover, in the notation
\[
\Gamma_{T(X)}(x)=\frac{h_T(x)}{m_T(x)},
\]
we have \(m_T=D_T, h_T=N_T,\) and hence \(m_T(0)=-m(0), h_T(-1)=-h(-1).\) Thus, if $m(0)h(-1)\neq0$, then also $m_T(0)h_T(-1)\neq0$, so the same non-vanishing conditions persist under iteration of $T$.
\end{proof}

For the seed $G_{10}$, Appendix~\ref{app:seed-verification} gives a reduced coronal
\[
\Gamma_{G_{10}}(x)=\frac{h_0(x)}{m_0(x)}
\]
with \(m_0(0)=4, m_0(-1)=-4, h_0(0)=-12, h_0(-1)=8.\) Thus Proposition~\ref{prop:rank-lifting} applies exactly as required.

\begin{lemma}\label{lem:noniso-operations}
If $X,Y$ are connected, non-isomorphic graphs of order at least $2$, then \(\mathcal T(X)\not\cong\mathcal T(Y).\) Moreover, $\mathcal U(G_{10})\not\cong\mathcal U(H_{10})$.
\end{lemma}

\begin{proof}
In \(\mathcal T(X)=(X\cup K_1)\vee K_1,\) the last added vertex is universal. Every vertex of $X$ is nonadjacent to the isolated vertex introduced before the join, while that isolated vertex is adjacent only to the last added vertex. Hence the last added vertex is the unique universal vertex. Removing it leaves $X\cup K_1$; since $X$ is connected and has at least two vertices, the added $K_1$ is the unique isolated component. Thus $X$ can be recovered from $\mathcal T(X)$ up to isomorphism. For $\mathcal U(G_{10})$ and $\mathcal U(H_{10})$, the new vertex is again uniquely universal because neither seed has a universal vertex. Deleting it would otherwise force $G_{10}\cong H_{10}$.
\end{proof}

\begin{proof}[Proof of Theorem~\ref{thm:order-dichotomy}]
The necessity $n\ge10$ is Proposition~\ref{prop:no-small-order}. It remains to produce examples for every $n\ge10$.

For even orders, set
\(
G_{10+2k}=\mathcal T^k(G_{10}),
H_{10+2k}=\mathcal T^k(H_{10}) (k\ge0).
\)
Lemma~\ref{lem:UT-reciprocal} and Proposition~\ref{prop:reciprocal-walks} give equal walk matrices. Each application of $\mathcal T$ adds a universal vertex, so the graphs remain connected. Proposition~\ref{prop:rank-lifting}, together with the seed values above, gives
\[
\operatorname{rank}W_{G_{10+2k}}=8+2k=(10+2k)-2,
\]
and the same holds for $H_{10+2k}$. Lemma~\ref{lem:noniso-operations} preserves non-isomorphism.

For odd orders, begin with \(G_{11}=\mathcal U(G_{10}), H_{11}=\mathcal U(H_{10}).\) These graphs are connected and non-isomorphic. By Lemma~\ref{lem:UT-reciprocal} and Proposition~\ref{prop:reciprocal-walks}, \(W_{G_{11}}=W_{H_{11}}.\) Moreover, the same equality remains valid after every subsequent application of $\mathcal T$. 

Since \(m_0(-1)+h_0(-1)=4\ne0,\) Proposition~\ref{prop:rank-lifting} gives walk rank $9$. The new reduced coronal satisfies
\[
m_1(0)=-h_0(0)=12,
\quad
h_1(-1)=m_0(-1)+h_0(-1)=4,
\]
so $\mathcal T$ can now be iterated. Define \(G_{11+2k}=\mathcal T^k(G_{11}), H_{11+2k}=\mathcal T^k(H_{11}).\) Their common walk rank is $9+2k=(11+2k)-2$; connectedness and non-isomorphism persist. This completes the proof.
\end{proof}

\section{Further Questions}
Two questions remain especially natural. First, suppose \(\operatorname{rank}W_G=n-r, r\ge 3,\) and \(W_G=W_H\). The proof of Lemma~\ref{lem:difference-rank-two} still gives
\(\operatorname{rank}(A_G-A_H)\le r.\) For \(r=2\), zero diagonal together with \(\{0,\pm1\}\)-entries forces the complete-bipartite sign pattern exploited above. At higher
rank that rigidity disappears. Is there nevertheless a finite list of switching mechanisms governing walk-matrix ambiguity for each fixed corank?

Second, there is a broader fixed-space problem behind the preceding arguments. More generally, if \(B=Q^TAQ, Q\mathbf e=\mathbf e,\) with $Q$ orthogonal, then \(W(B)=Q^TW(A).\) Hence $W(B)=W(A)$ if and only if $Q$ fixes the Krylov space \(\mathcal K_G=\operatorname{col}W(A)\) pointwise. This suggests studying regular orthogonal transformations
that fix a prescribed Krylov space and, at the same time, conjugate an adjacency matrix to another adjacency matrix. Such a viewpoint may provide a natural framework for walk-matrix equivalence beyond corank two. Recent work on regular-orthogonal switching indicates that substantially richer behaviour should occur~\cite{MaoWangLiuQiu2023,AbiadVanDeBergSimoens2026}.

\appendix

\section{Exact verification for the order-\texorpdfstring{\(10\)}{10} seed}\label{app:seed-verification}

This appendix gives exact certificates for the computations used in Sections~\ref{sec:minorder} and~\ref{sec:families}. Let \(A=A(G_{10})\), where \(G_{10}\) is the graph defined in Section~\ref{sec:minorder}, and let \(\mathbf e\) be the all-one vector.

Set \(\chi_{10}(x):=\det(xI-A).\) A direct determinant calculation gives
\[
\chi_{10}(x)=(x^2+x-4)
(x^8-x^7-13x^6-x^5+38x^4+12x^3-29x^2-5x+4).
\]
Thus \(\chi_{10}(x)=(x^2+x-4)m_0(x),\) where
\[
m_0(x)=x^8-x^7-13x^6-x^5+38x^4+12x^3-29x^2-5x+4.
\]

To compute the coronal, use the matrix determinant lemma:
\[
\det(xI-A+\mathbf e\mathbf e^T)-\det(xI-A)=
\det(xI-A)\,
\mathbf e^T(xI-A)^{-1}\mathbf e.
\]
The left-hand side factors as
\[
\begin{aligned}
&\det(xI-A+\mathbf e\mathbf e^T)-\det(xI-A)\\
&\quad =
(x^2+x-4)
\bigl(
10x^7+26x^6-30x^5-102x^4
+2x^3+86x^2+8x-12
\bigr).
\end{aligned}
\]
Hence
\[
    \Gamma_{G_{10}}(x)
    =
    \frac{h_0(x)}{m_0(x)},
\]
where
\[
    h_0(x)
    =
    10x^7+26x^6-30x^5-102x^4
    +2x^3+86x^2+8x-12.
\]

The fraction is reduced. Indeed, the resultant of $m_0$ and $h_0$ satisfies
\[
\operatorname{Res}(m_0,h_0)=36864\neq0,
\]
and therefore \(\gcd(m_0,h_0)=1.\)  In particular,
\[
m_0(0)=4,\quad
m_0(-1)=-4,\quad
h_0(0)=-12,\quad
h_0(-1)=8,
\]
which are precisely the values used in the rank-lifting argument.

The rank certificate used in Proposition~\ref{prop:order-ten} is as follows. The submatrix of
\[
[e,Ae,\ldots,A^7e]
\]
formed by the rows
\[
1,2,3,5,6,7,9,10
\]
is
\[
M=
\begin{pmatrix}
1&5&18&71&266&1014&3829&14511\\
1&3&10&37&138&520&1965&7427\\
1&4&13&50&184&699&2631&9965\\
1&4&16&60&228&862&3262&12346\\
1&2&8&28&108&404&1534&5794\\
1&3&12&44&169&637&2417&9142\\
1&4&16&59&225&845&3204&12101\\
1&4&16&61&233&883&3348&12671
\end{pmatrix}.
\]
A direct calculation gives $\det M=48\neq0$. Therefore \(\operatorname{rank}[e,Ae,\ldots,A^7e]\ge 8,\) and hence \(\operatorname{rank}W_{G_{10}}\ge 8.\) On the other hand, Proposition~\ref{prop:reciprocal-walks} gives \(\operatorname{rank}W_{G_{10}}\le 8.\) Consequently, \(\operatorname{rank}W_{G_{10}}=8.\)

\section*{Acknowledgments}
We sincerely thank Professor Wei Wang of Anhui Polytechnic University for his valuable comments and suggestions.

\end{document}